\documentclass[11pt,a4paper]{article}
\usepackage[T1]{fontenc}
\usepackage{lmodern}
\usepackage[margin=1in]{geometry}
\usepackage{amsmath,amssymb,amsthm,amsfonts,bm,float,cite}
\newtheorem{conjecture}{Conjecture}
\newtheorem{theorem}{Theorem}[section]

\newtheorem{definition}{Definition}[section]
\newtheorem{algorithm}{Algorithm}[section]
\usepackage[center]{caption}
\usepackage{xcolor}
\definecolor{codegreen}{rgb}{0,0.6,0}
\definecolor{codegray}{rgb}{0.5,0.5,0.5}
\definecolor{codepurple}{rgb}{0.58,0,0.82}
\definecolor{backcolour}{rgb}{0.95,0.95,0.92}

\definecolor{mygray}{rgb}{0.42,0.42,0.42}
\usepackage{upquote}
\usepackage{listings}
\usepackage[hidelinks]{hyperref}
\hypersetup{
  pdftitle={Perfect numbers, Wieferich primes and the solutions of binomial (2n,n) congruent to 2 to the n modulo n},
  pdfauthor={Gabriel Guedes and Ricardo Machado},
  pdfsubject={Number theory; published in Notes on Number Theory and Discrete Mathematics 29 (2023), 705--712},
  pdfkeywords={Perfect numbers, Wieferich primes, binomial coefficient, algorithmic number theory, Wolstenholme converse problem}
}

\title{Perfect numbers, Wieferich primes and the solutions of
\texorpdfstring{$\displaystyle\binom{2n}{n}\equiv 2^n \bmod n$}
{binomial (2n,n) congruent to 2^n modulo n}}
\author{
Gabriel Guedes\thanks{Department of Mathematics, Federal Rural University of Pernambuco (UFRPE), Recife-PE, Brazil. Email: \href{mailto:gabriel.guedes@ufrpe.br}{gabriel.guedes@ufrpe.br}.}
\and
Ricardo Machado\thanks{Department of Mathematics, Federal Rural University of Pernambuco (UFRPE), Recife-PE, Brazil. Email: \href{mailto:ricardo.machadojunior@ufrpe.br}{ricardo.machadojunior@ufrpe.br}.}
}
\date{}

\begin{document}
\maketitle

\begin{abstract}
In this article we focus on the solutions of a congruence equation: ``$\binom{2n}{n}\equiv 2^n \bmod n$".
Using the main result of this article and the SageMath software, we improve largely the number of known solutions. Furthermore, we prove that some famous numbers like even perfect numbers and Wieferich primes are connected to solutions of this equation.
\end{abstract}



\medskip
\noindent\textit{Published in \emph{Notes on Number Theory and Discrete Mathematics} 29 (2023), no.~4, 705--712. DOI: \href{https://doi.org/10.7546/nntdm.2023.29.4.705-712}{10.7546/nntdm.2023.29.4.705-712}. This work is licensed under the \href{https://creativecommons.org/licenses/by/4.0/}{Creative Commons Attribution 4.0 International License}.}
\vspace{4mm}

\section{Introduction} \label{sec:Intr}

 The search for properties that characterize prime numbers is intense, given their fundamental importance. 

 A classical theorem in number theory is the Wolstenholme's Theorem.

\smallskip

\begin{theorem}[Wolstenholme]\label{Wolstenholme}
Let $p>3$ be a prime number, then $p^3\mid\binom{2p}{p}-2.$
\end{theorem}

\smallskip

A natural question is the reciprocal conclusion of the Theorem \ref{Wolstenholme}.

\smallskip

\begin{conjecture}\label{conj01}
If $n>3$ and $n^3\mid\binom{2n-1}{n-1}-1$ then $n$ is a prime number.
\end{conjecture}

The Conjecture \ref{conj01} is known as Wolstenholme's converse problem which if true will characterize the prime numbers. This conjecture was proposed by J. P. Jones and remains open. See \cite[p. 23]{Ribenboim} and \cite[p. 84]{Guy}. In \cite{McIntsh}, McIntosh verified that this holds for all $n<10^9.$

We are interested in problems related to Conjecture \ref{conj01}. Using the Theorem \ref{Wolstenholme} and Fermat's Little Theorem we arrive at the congruence equation: $\binom{2p}{p}\equiv 2^p \bmod {p} $, which is valid for primes. We can ask for solutions to the same congruence equation in the case of composite integers. That is, which are the composite integers $n$ that satisfy the following equation?
\begin{equation*}
\binom{2n}{n}\equiv 2^n \bmod n.
\end{equation*}
In \cite{oeis} sequence id:A084699 shows that the known solutions of this equation are 
\begin{align*}
12, 30, 56, 424, 992, 16256, 58288, 119984,
356992, 1194649, \quad \\
 9973504, 12327121, 13141696,
22891184, 67100672, 233850649.
\end{align*}
In the next sections, we make use of Theorem \ref{thm111} and computational effort to find some new solutions and we were able to relate the set of solutions with some famous numbers such as Mersenne primes, even perfect numbers, and Wieferich primes.

\section{Solutions and perfect numbers}
The relation of the Equation \eqref{eq-001} and the very famous class of numbers as even perfect numbers, Mersenne primes, and Wieferich primes, appears to us in a very interesting way. Our initial aim with this article was to obtain new values to the sequence id:A084699 (see \cite{oeis}). For this, we use the result proved in this article, Theorem \ref{thm111}. In particular, item \textit{ b)} enabled a more efficient implementation than the direct test on Equation \eqref{eq-001}.

In this article, we show an algorithm implemented in SageMath, in order to generate solutions to this equation. 
The algorithm is based on the following theorem:

\begin{theorem}\label{thm111}
Let $n=2^k p$, where $p$ is an odd prime and $k\in \mathbb{N}$, then $n$ is solution of the equation
\begin{equation} \label{eq-001}
\binom{2n}{n} \equiv 2^n \bmod n
\end{equation}
if and only if $p$ satisfies the following conditions:
\begin{itemize}
    \item [a)] $p$ divides $\binom{2^{k+1}}{2^k}-2^{2^k}$;
    \item [b)] $p$ has at least $k$ digits $1$'s in its binary expansion.
\end{itemize}
\end{theorem}

\begin{proof}
Since $\gcd(2^k, p)=1$, the Equation~\eqref{eq-001} is equivalent to the system:
\begin{align}
\binom{2n}{n}&\equiv 2^{n} \bmod {p} \label{eq001.1},  \\ 
\binom{2n}{n}&\equiv 2^n \bmod {2^k}\label{eq002}.
\end{align}

Using the left-hand side of Equation \eqref{eq001.1}, by Babbage's Theorem (see page 68, Exercise 2.5.10 of \cite{GRANVILLE}), we have
$$\binom{2n}{n}\equiv \binom{2\cdot 2^k p}{2^k p}\equiv \binom{2^{k+1}}{2^k} \bmod {p}. $$  
Using the right-hand side of Equation \eqref{eq001.1}, by Fermat's Little Theorem,
$$2^n\equiv \left(2^{2^k}\right)^p\equiv 2^{2^k} \bmod {p}.$$
Consequently, the Equation \eqref{eq001.1} is equivalent to
$$\binom{2^{k+1}}{2^k} \equiv 2^{2^k} \bmod {p}, $$
which is equivalent to item  \textit{a)}.  
Now, using the right-hand side of Equation \eqref{eq002}, since $n=2^k p$, \linebreak it is clear that $2^n\equiv 0 \bmod {2^k}$.
Let $v_p(n)$ the $p$-adic valuation of $n$.
Since $v_2\left(\binom{2n}{n}\right)=v_2\left(\binom{2p}{p}\right)$, by Kummer's Theorem  (see \cite{GRANVILLE}, p. 99, Theorem 3.7), $v_2\left(\binom{2n}{n}\right)$ is the number of $1$'s in binary expansion of $p$. Then the Equation \eqref{eq002} is satisfied if and only if item \textit{b)} is satisfied, i.e., $p$ has at least $k$ digits $1$'s in binary expansion.
\end{proof}

In addition to the solutions available at \cite{oeis} (sequence id:A084699),
using Theorem \ref{thm111} and some implementation on SageMath, we discovered two more numbers: 
$$2^{17} \times 131071 = 17179738112 \quad \text{ and } \quad 2^{19} \times 524287 = 274877382656.$$ 
The algorithm code follows:
\begin{algorithm}\label{algo01} \rm \footnotesize
\mbox{ }
\begin{lstlisting}
def algorithm_theo2(start, end):
    '''spbin is the sum of the digits of p in base 2'''
    p = start
    while (p<=end):
        p = next_prime(p)
        spbin = sum(p.digits(base=2)) 
        for k in range(1,spbin+1):
            if (binomial(2^(k+1),2^(k))-2^(2^k)).mod(p)==0:
                print('2^%d x %d = %d' %(k, p, (2^k)*p))
algorithm_theo2(2, 10^6)
\end{lstlisting}
\end{algorithm}
The Algorithm \ref{algo01} searches for primes $p$ such that 
$$ \binom{2^{k+1}}{2^k}-2^{2^k} \equiv 0 \bmod p, $$
in which $k$ varies from $1$ to the sum of the digits of $p$ in base $2$. In all computational tests, we used a PC with a Core i5 1135G7 processor and this routine took 55 seconds to compute.

Factoring the solutions found, we verify that the Mersenne primes are related to a solution in the way of the next theorem.

\begin{theorem} \label{cor-perfect-number}
If $m$ is an even perfect number, then $n=2m$ satisfies the Equation \eqref{eq-001}.
\end{theorem}

\begin{proof}
If $m$ is an even perfect number, then $m=2^{p-1}\cdot q$ with $p$ and $q=2^p-1$ prime numbers. Like the proof of Theorem \ref{thm111}, Equation \eqref{eq-001} is equivalent to the system
\begin{align}
\binom{2^{p+1}}{2^{p}}&\equiv 2^{2^p} \bmod {q}  \label{eq5}, \\ 
\binom{2q}{q}&\equiv 2^q \bmod {2^p}\label{eq6}.
\end{align}
Now, we consider two cases. \\
Case 1: If $p=2$, the congruence is immediate.

\noindent
Case 2: If $p>2$, the $q$-adic expansion of $2^p=2^p-1+1=q+1$ and $2^{p+1}=2q+2$ by Lucas's Theorem (see \cite{CHENCHUAN}, page 95) the left-hand side of Equation \eqref{eq5}
\begin{equation*}
    \binom{2^{p+1}}{2^{p}}\equiv \binom{2q+2}{q+1}\equiv \binom{2}{1}\cdot\binom{2}{1}\equiv 4 \bmod {q}. 
\end{equation*}
Now look to the right-hand side and applying Fermat's Little Theorem
\begin{equation*}
 2^{2^p}=2\cdot 2^{2^p-1}=2\cdot2^q\equiv4 \bmod {q}    
 \end{equation*}
and this verifies the Equation \eqref{eq5}.

In Equation \eqref{eq6}, applying Kummer's Theorem for the left-hand side  we have $v_2\left(\binom{2q}{q}\right)=p$, then $2^p  \mid   \binom{2q}{q}$. For the right-hand side, $p<q\Rightarrow2^p \mid 2^q$. Hence $2^p$ divides both sides of the Equation~\eqref{eq6}.
\end{proof}

\subsection{Using Pollard's Rho algorithm for prime factorization}

Pollard's Rho (see \cite{CORMEN}) algorithm for prime factorization is particularly fast for a large composite number with small prime factors.   Algorithm \ref{algo02} uses Pollard’s Rho algorithm to search for some prime factors of $\binom{2^{k+1}}{2^k}-2^{2^k}$ and proceeds like Algorithm \ref{algo01} to verify the second condition of Theorem \ref{thm111}. Using Algorithm \ref{algo02}, we found 8 previously unknown solutions to Equation \eqref{eq-001}.

\begin{algorithm}\label{algo02} \rm \footnotesize
\mbox{ }
\begin{lstlisting} 
def algorithm_theo2_pollard_rho(start, end, pmax):
    '''spbin is the sum of the digits of p in base 2'''
    k = start
    while (k<=end):
        rk = (binomial(2^(k+1),2^(k))-2^(2^k))
        prime_list=partial_factoring(rk, pmax)
        for p in prime_list:
            if rk%p ==0:
                spbin = sum(ZZ(p).digits(base=2)) 
                if spbin >= k:
                    print('2^%d x %d = %d' %(k, p, (2^k)*p))
        k=k+1
\end{lstlisting} \normalsize
\end{algorithm}
Line 6 of Algorithm \ref{algo02} calls the function \verb|partial_factoring| that can be found in Appendix \ref{append}. This function uses the method \verb|is_prime| which employs a strong pseudo-primality test. 
Using the Algorithm \ref{algo02}, we found the following previously unknown solutions to Equation~\eqref{eq-001}:
\begin{enumerate}
\setlength\itemsep{-0.1em} 
    \item $2^5 \times 29558453816897149,$
    \item $2^6 \times 52917841,$
    \item $2^6 \times 41811313718690881087,$
    \item $2^7 \times 326653838319686945891577906833524101155696807553436546433716230841,$
    \item $2^8 \times 129020293739,$
    \item $2^9 \times 1466765681,$
    \item $2^{10} \times 412707874957531,$
    \item $2^{11} \times 559115197117.$
\end{enumerate}
To find the first 6 solutions we used parameters \verb|start=3|, \verb|end=9| and \verb|pmax=10^8|. That is, we used the command: \verb|algorithm_theo2_pollard_rho(3, 9, 10^8)|.

It was necessary to use 2.1 seconds to get the solutions. 
To find the last two solutions required 3 minutes and 33 seconds. The parameters were: \verb|start=10|, \verb|end=11| and \verb|pmax=10^4|.

\section{A generalization and Wieferich primes}

We can observe that it is possible to generalize Theorem \ref{thm111} for any positive integer, in the next Theorem \ref{thm444}. It is noteworthy that using this theorem, we show that the Wieferich prime numbers are uniquely determined by Equation~\eqref{eq-001}.

\begin{theorem}\label{thm444}
Let $n=2^k p_1\cdots p_s$, where $p_j$ are different prime numbers and $n_j=\frac{n}{p_j}$, then $n$ is solution of the equation
\begin{equation}
\binom{2n}{n} \equiv 2^n \bmod n 
\end{equation}
if and only if the following conditions  
are satisfied:
\begin{itemize}
    \item [a)] $p_j$ divides $\binom{2n_j}{n_j}-2^{n_j}$;
    \item [b)] $p_1\cdots p_s$ has at least $k$ digits $1$'s in its binary expansion.
\end{itemize}
\end{theorem}
\begin{proof}
Is the same as in Theorem \ref{thm111} applied to this more general framework.
\end{proof}

\begin{definition}
A {\em Wieferich prime} is a prime number $p$ such that $p^2$ divides $2^{p-1}-1$.
\end{definition}

\smallskip

The Wieferich primes were defined in 1909  by Arthur Wieferich in his studies of Fermat's Last Theorem.  There is a lot of connection with this class of primes and important topics in number theory like {\it ABC} conjecture, pseudoprimes,  Mersenne primes, and others. Until today, in March of 2023, there are only two known Wieferich primes, $1093$ and $3511$.
\begin{theorem}\label{teo-32}
Let $n=p^2$ with $p$ a prime number. We have that
$p$ is a Wieferich prime if and only if  $n$ is a solution of Equation~\eqref{eq-001}.
\end{theorem}

\begin{proof}
Supposing that $p$ is a Wieferich prime we are going to prove that $n$ is a solution of Equation \eqref{eq-001}, that is,
\begin{equation}\label{eq8}
\binom{2p^2}{p^2}\equiv2^{p^2} \bmod {p^2}.
\end{equation}
Using Lucas's Theorem  we have that left-hand side of Equation \eqref{eq8} is
$$\binom{2p^2}{p^2}\equiv\binom{2}{1}\equiv 2 \bmod {p^2}.$$ In the right-hand side of \eqref{eq8} note that
$2^{p^2}=2^{p(p-1)}\cdot2^p$,
by Euler's Theorem 
$2^{p(p-1)}\equiv 1 \bmod {p^2}$
then 
$$2^{p^2}\equiv2^{p(p-1)}\cdot2^p\equiv2^p\equiv2^p -2 +2\equiv 2\cdot(2^{p-1}-1) +2 \bmod {p^2}.$$
Now apply the hypothesis of $p$ being a Wieferich prime then
$$2\cdot(2^{p-1}-1) +2\equiv 2 \bmod {p^2.}$$
Thus, we show that Equation \eqref{eq8} is valid, since both sides are congruent to $2$ modulo $p^2$.

For the converse in the same lines we have that
$2^p\equiv 2 \bmod p^2$
simplifying by $2$ we obtain
$2^{p-1}\equiv 1 \bmod {p^2}$
which 
is the definition of a Wieferich prime.
\end{proof}

One natural question arose here: \textit{How many of the known solutions of Equation \eqref{eq-001} are explained by the theorems of this article?} In Table \ref{tab01} we can see that there is only one previously known solution ($n=233850649$) that cannot be explained by the theorems of this article. All the new solutions found in this article are of the form $2^kp$, where $p$ is a prime number. Hence they are explained by Theorem~\ref{thm111}.
\begin{table}[H]
\centering
\begin{tabular}{ |p{2cm}|p{3cm}|p{1cm}|p{3cm}|p{2.5cm}|  }
\hline \hline
\textbf{Solution} & \textbf{Solution factorization} & \bm{$k$} & \textbf{Ones in bin. exp. of} \bm{$p_1\cdots p_s$}  & \textbf{Theorem} \\
\hline \hline
$12$ & $2^2 \cdot 3$ & $2$ & $2$ & Theorem 2.1 \\
$30$ & $ 2 \cdot 3 \cdot 5$ & $1$ & $4$ & Theorem 3.1 \\
$56$ & $2^3 \cdot 7$ & $3$ & $3$ & Theorem 2.1 \\
$424$ & $2^3 \cdot 53$ & $3$ & $4$ & Theorem 2.1 \\
$992$ & $2^5 \cdot 31$ &$5$ & $5$ & Theorem 2.1 \\
$16256$ & $2^7 \cdot 127$ &$7$ & $7$ & Theorem 2.1 \\
$58288$ & $2^4 \cdot 3643$ & $4$ & $8$& Theorem 2.1 \\
$119984$ & $2^4 \cdot 7499$ & $4$ & $8$ & Theorem 2.1 \\
$356992$ & $2^7 \cdot 2789$ & $7$ & $7$ & Theorem 2.1 \\
$1194649$ & $1093^2$ &$-$&$-$& Theorem 3.2 \\
$9973504$ & $2^8 \cdot 38959$ & $8$ & $8$ & Theorem 2.1 \\
$12327121$ & $3511^2$ &$-$&$-$& Theorem 3.2 \\
$13141696$ & $2^6 \cdot 205339$ & $6$ & $8$ & Theorem 2.1 \\
$22891184$ & $2^4 \cdot 607 \cdot 2357$ & $4$ & $12$ & Theorem 3.1 \\
$67100672$ & $2^{13} \cdot 8191$ & $13$ & $13$ & Theorem 2.1 \\
$233850649$ & $3919 \cdot 59671$ &$-$&$-$& ? \\
\hline \hline
\end{tabular}
\caption{Prior known solutions of \eqref{eq-001} explained by
Theorems \ref{thm111} and \ref{teo-32}.}\label{tab01}
\end{table}

\section{Conclusion}
In this article, we describe algorithms that facilitate the search for solutions to Equation \eqref{eq-001}, and with it, we were able to find new solutions. Furthermore we related the set of solutions with even perfect numbers and Wieferich primes. Much remains to be done, we describe solutions of the form $n=2^kp_1p_2\cdots p_r$ or $n=p^2$. However, there is a lot to wonder about solutions with exponents greater than 1.

\section{Appendix} \label{append}

In this Appendix we have the functions that are called, directly or indirectly, by Algorithm \ref{algo02}. 

\begin{algorithm}\label{algo03} \rm \footnotesize
\mbox{ }
\begin{lstlisting}
def pollard_rho_ref(n):
    i=1
    x=randint(0, n-1)
    y=x
    k=2
    while True:
        i=i+1
        x=(x^2-1)%n
        d=gcd(y-x, n)
        if d!= 1 and d!=n:
            if d.is_prime():
                return d
            else:
                f = list(d.factor())[0][0]
                return f
        elif i==k:
            y=x
            k=2*k
\end{lstlisting}
\end{algorithm}
The Algorithm \ref{algo03} is a SageMath adaptation of the pseudocode of page 976 of \cite{CORMEN}. The differences here are in lines 11 to 15. The original pseudocode prints the value of $d$, while the Algorithm \ref{algo03} returns a prime that divides $d$.
 
\begin{algorithm}\label{algo004} \rm \footnotesize
\mbox{ }
\begin{lstlisting}
def partial_factoring(n, max=oo):
    prime_list=[]
    while n!=1:
        if n.is_prime():
            prime_list.append(n)
            return prime_list
        else:
            p=pollard_rho_ref(n)
            while n%p==0:
                n=ZZ(n/p)
            prime_list.append(p)
            if p>=max:
                return prime_list
\end{lstlisting}
 \end{algorithm}
The Algorithm \ref{algo004} tries to create a divisor list of \verb|n|, it has the inputs \verb|n| and \verb|max|. This algorithm stops adding the primes to the list when it finds all the primes of \verb|n| or when it finds a prime number greater than the input \verb|max|. This algorithm executes Algorithm \ref{algo03} to search for primes, divides \verb|n| by the prime number found and runs the Algorithm~\ref{algo03} again with the new value.

\section*{Acknowledgements} 

We would like to thank Paulo Ribenboim, for his books, his lecture, and his entire life dedicated to mathematics. To the Department of Mathematics of the Federal Rural University of Pernambuco - UFRPE, for the excellent work environment, with the appreciation of the broad spectrum that
academic life can have. 

Also, we would like to express our gratitude to the editors and anonymous referees for their feedback including the suggestion to use Pollard’s Rho Algorithm to improve our algorithm and find more solutions to Equation \eqref{eq-001}. Our sincere thanks.

\makeatletter
\renewcommand{\@biblabel}[1]{[#1]\hfill}
\makeatother

\end{document}